\documentclass[12pt,leqno,amsfonts,amscd]{amsart}
\usepackage{epsfig}
\usepackage{amsmath}
\usepackage{amsfonts}
\usepackage{amssymb}
\usepackage{color}
\usepackage{hyperref}

\newtheorem{theorem}{Th\'eor\`eme}[section]
\newtheorem{prop}[theorem]{Proposition}
\newtheorem{cor}[theorem]{Corollaire}
\newtheorem{lemma}[theorem]{Lemme}
\theoremstyle{definition}

\theoremstyle{remark}

\newtheorem{remark}[theorem]{Remarque}

\numberwithin{equation}{section}

\newcommand{\NN}{{\mathbb N}}

\newcommand{\RR}{{\mathbb R}}

\newcommand{\out}[1]{\ }

\let\cal=\mathcal
\renewcommand{\phi}{\varphi}

\begin{document}

\title [Propri\'et\'e de convergence
et r\'egularit\'e du noyau de Green fin]{Propri\'et\'e de convergence
de certaines familles de fonctions finement harmoniques et r\'egularit\'e du noyau de Green d'un domaine fin}

\author{Abderrahim Aslimani}

\author{Mohamed El Kadiri}
\address{Universit\'e Mohammed V
\\D\'epartemnt de Math\'ematiques
\\Facult\'e des Sciences
\\B.P. 1014, Rabat
\\Morocco}
\email{slimonier.math.@gmail.com}
\email{elkadiri@fsr.ac.ma}

\maketitle

{\bf R\'esum\'e.}
Nous d\'emontrons une propri\'et\'e de convergence de certaines
familles de fonctions finement harmoniques dans un domaine fin $U$ de
$\RR^n$ ($n\ge 2$) et nous l'appliquons pour \'etablir
une certaine r\'egularit\'e du noyau de Green fin de $U$.


\section{Introduction.}
On se place dans un domaine de Green  $\Omega$ de $\RR^n$, c'est-\`a-dire un domaine
r\'egulier quelconque de $\RR^n$ si $n\ge 3$, ou un domaine de compl\'ementaire non
polaire si $n=2$. Rappelons
que le noyau (ou fonction) de Green de $\Omega$ est une fonction sym\'etrique $G$ d\'efinie sur $\Omega\times \Omega$ \`a valeurs
dans $]0,+\infty[$ ayant les propri\'et\'es suivantes:

1. $G$ est s.c.i. sur $\Omega\times\Omega$ et continue
en dehors de la diagonale de $\Omega\times \Omega$.

2. Pour tout $y\in \Omega$, la fonction $G(.,y)$ est un potentiel,
harmonique dans $\Omega\setminus \{y\}$.

Toute autre fonction $G'$ sur $\Omega\times \Omega$ \`a valeurs
dans $]0,+\infty[$ poss\'edant les propri\'et\'es 1. et 2. est de la
forme $G'(.,y)=\varphi(y)G(.,y)$ pour tout
$y\in \Omega$, o\`u $\varphi$ est une fonction finie continue et $>0$ sur $\Omega$.
De plus, comme $G$ est sym\'etrique, la fonction $G'$ est sym\'etrique si et seulement si
$\varphi$ est constante.
Si $\Omega=\RR^n$, $n\ge 3$, alors $G$ est donn\'e par
$G(x,y)=\frac{1}{||x-y||^{n-2}}$ \`a la multiplication pr\`es par une
constante $>0$.
La notion de fonction ou noyau de Green a \'et\'e \'etendue au cadre g\'en\'eral des
espaces harmoniques de Brelot v\'erifiant l'hypoth\`ese d'unicit\'e par R.-M. Herv\'e
dans \cite{He}.

Soit $U$ un domaine fin de $\Omega$, c'est-\`a-dire un domaine
au sens de la topologie fine sur $\Omega$. Rappelons que la topologie fine,
d\'efinie par Cartan en 1940, est la moins fine des topologies
sur $\Omega$ qui rendent continues les fonctions
surharmoniques dans $\Omega$ (pour plus
de d\'etails sur cette topologie on renvoie \`a
\cite[Chapter 1]{F1} et \cite{DO}). Pour tout $y\in U$, on note
$G_U(.,y)$  la fonction d\'efinie sur $U\setminus \{y\}$ par
$G_U(.,y)=G(.,y)-\widehat R_{G(.,y)}^{\complement U}$. Cette fonction
est finement surharmonique $\ge 0$ dans $U\setminus \{y\}$,
et le point $y$ est polaire, donc elle se prolonge par continuit\'e fine
\`a $U$ en une fonction finement surharmonique sur $U$, not\'ee encore $G_U(.,y)$.
La fonction $(x,y)\mapsto G_U(x,y)$ d\'efinie sur $U\times U$ est appel\'ee un noyau de
Green fin de $U$. D'apr\`es \cite[Th\'eor\`eme, p. 203]{F0}, pour tout $y\in U$,
$G_U(.,y)$ est un potentiel fin dans $U$ et tout potentiel fin
dans $U$ finement harmonique dans $U\setminus \{y\}$ est de la forme
$\alpha(y)G_U(.,y)$, o\`u $\alpha(y)$ est une constante $>0$ ne d\'ependant que de
$U$ et de $y$.

On peut alors se demander si le noyau de Green  fin $G_U$ de $U$ est
r\'egulier dans le sens o\`u  $G_U$ finement s.c.i. dans $U\times U$ et finement continue en dehors de la diagonale de $U\times U$.
Dans le cas classique d'un domaine euclidien de $\RR^n$, la propri\'et\'e
2. d'un noyau de Green de $\Omega$ est une cons\'equence du principe de
Harnack pour les fonctions harmoniques positives dans $\Omega$. Or le principe de
Harnack n'est pas satisfait par les fonctions finement harmoniques dans un domaine fin.
Toutefois ces fonctions poss\`edent une propri\'et\'e de convergence voisine.
C'est cette propri\'et\'e qui va nous permettre de r\'epondre \`a la question
soulev\'ee ci-dessus. Plus pr\'ecisemment, nous \'etablissons une propri\'et\'e de
convergence pour les suites et les familles de fonctions finement harmoniques uniform\'ement
finement localement born\'ees, et gr\^ace \`a cette propri\'et\'e et \`a l'aide d'une comparaison entre
la topologie naturelle  et la topologie fine de $U$, nous montrons que
le noyau $G_U$ est r\'egulier.

Les r\'esultats de ce travail sont valables dans le cadre
g\'en\'eral d'un $\cal P$-espace harmonique de la th\'eorie axiomatique de Brelot
\`a base d\'enombrable qui satisfait l'axiome (D), l'hypoth\`ese d'unicit\'e et dans lequel la topologie
fine est moins fine que la topologie fine adjointe, pour de tels
espaces on renvoie \`a  \cite{He}. On s'est plac\'e dans un domaine de Green
de l'espace $\RR^n$ pour des raisons de simplicit\'e seulement.

{\bf Notations et d\'efinitions}:
Dans tout ce travail on se place dans un domaine r\'egulier $\Omega$ de $\RR^n$, de compl\'ementaire non polaire
si $n=2$, et on consid\`ere un domaine fin $U$ de $\Omega$, c'est \`a-dire au sens de la topologie
fine de $\Omega$. Le noyau de Green de $\Omega$ est not\'e tout simplement
$G$. Nous utilisons le mot fin (finement)
pour distinguer les notions relatives
\`a la topologie fine de celles relatives \`a la topologie
euclidienne de $U$, c'est-\`a-dire la topologie induite
sur $U$ par la topologie usuelle de $\RR^n$. Pour toute partie $A$ de $U$, on note $\overline A$
l'adh\'erence de $A$ dans $\Omega$ en topologie euclidienne
et $\tilde A$ l'adh\'erence de $A$ dans $\Omega$ en topologie fine.
On note ${\rm f}$-$\lim$ et ${\rm f}$-$\liminf$ la limite et la limite inf
en topologie fine. On
note aussi  ${\cal S}(U)$ le c\^one convexe des fonctions finement
surharmoniques $\ge 0$ dans  $U$ au sens de \cite{F1}.
Si $f$ est une fonction sur $U$ \`a valeurs dans $\overline \RR$, on note
$\widehat f$ la r\'egularis\'ee finement s.c.i de $f$, c'est-\`a-dire la plus grande minorante
finement s.c.i de $f$. Le noyau de Green fin de
$U$ est not\'e $G_U$. Pour plus de d\'etails sur ce noyau nous
renovoyons \`a \cite{F0}.

\section{ Propri\'et\'e de convergence de familles et de suites de fonctions finement harmoniques}

D'apr\`es \cite[Th\'eor\`eme 3.1, p. 107]{EK1}, il existe une r\'esolvante
absolument continue $(V_\lambda)$ de noyaux
bor\'eliens sur $U$ dont le c\^one des fonctions
excessives  finies $(V_\lambda)$-p.p. est le c\^one $\cal S(U)$. Il en r\'esulte d'apr\`es \cite[Theorem 4.4.6, p. 136]{BBC}
que $\cal S(U)$ est un $H$-c\^one standard
de fonctions. On le munit alors de la topologie naturelle \cite[Section 4.5, p. 141]{BBC}.
Rappelons que cette topologie
est induite sur $\cal S(U)$ par celle d'un espace vectoriel localement convexe $E$ dans
lequel $\cal S(U)$ est un c\^one convexe saillant
et que, pour cette topologie, si un filtre $\cal F$ sur $\cal S(U)$ est convergent, alors
on a $\lim_{\cal F}=\sup_{M\in \cal F}{\widehat {\inf}}_{u\in M}u$ (\cite[Theorem 4.5.2]{BBC}).
De plus, pour cette topologie $\cal S(U)$ est localement compact et admet une base
compacte \cite[Corollaire 2, p. 110]{EK1}.

\begin{lemma}\label{lemma2.1}
Soit $x_0\in U$. Alors il existe un voisinage fin
de $x_0$ compact $V\subset U$ tel que la r\'estriction
de toute fonction $u\in \cal S(U)$ \`a $V$
est s.c.i (en topologie euclidienne).
\end{lemma}

{\it D\'emonstration.} Comme $\cal S(U)$ est un $H$-c\^one standard
de fonctions, il existe d'apr\`es \cite[Definition, p. 104 et Theorem 4.4.6]{BBC} une suite
$(s_n)$ de fonctions de $\cal S(U)$ telle que toute
fonction $s\in \cal S(U)$ est l'enveloppe sup\'erieure d'une sous-suite de $(s_n)$.
Soit $x_0\in U$, il existe d'apr\`es \cite[Lemma, p. 114]{F5} un voisinage fin $V$
de $x_0$,
compact en topologie euclidienne,  tel que la r\'estriction de
toute fonction $s_n$ \`a $V$ est continue en topologie euclidienne.
On en d\'eduit donc que la r\'estriction de toute fonction $s\in \cal S(U)$ \`a $V$
est s.c.i. en topologie euclidienne.

\begin{theorem}\label{thm2.2}
Soit $(h_i)_{i\in I}$ une famille  de fonctions finement harmoniques $\ge 0$
uniform\'ement  finement localement born\'ee dans $U$,
et soit $\cal F$ un filtre sur $I$. Supposons que la famille $(h_i)$ converge selon $\cal F$
vers une fonction $h\in \cal S(U)$ en topologie
de $\cal S(U)$.
Alors   $h$ est finement harmonique dans $U$ et la famille $(h_i)$ converge, selon $\cal F$,
uniform\'ement finement localement
dans $U$ vers $h$.
\end{theorem}

{\it D\'emonstration.} Quitte \`a se placer finement localement
et ajouter une constante on peut supposer que $0\le h_i\le c$
dans $U$ pour
une  certaine constante $c>0$ et pour tout
$i\in I$. On a $\lim_{i,\cal F}h_i=\sup_{M\in \cal F}{\widehat {\inf_{i\in M}}}h_i=h\in \cal S(U)$
d'apr\`es \cite[Theorem 4.5.2]{BBC}.
D'autre part on a $c=c-h_i+h_i$, et $\lim_{i,\cal F}h_i=h$ dans $\cal S(U)$,
donc $\lim_{i,\cal F}(c-h_i)=c-h$ dans $\cal S(U)$.
Soit $x_0\in U$,
alors d'apr\`es le lemme pr\'ec\'edent,
il existe un voisinage fin $V$ de $x_0$, compact en topologie euclidienne, sur lequel les fonctions
$\widehat{\inf_{i\in M}}h_i$, $M\in \cal F$, sont s.c.i en topologie euclidienne
et $h$ est continue en topologie euclidienne.
Soit $\epsilon >0$, comme $V$ est compact
on peut trouver $M_0\in \cal F$ tel que
$h-\widehat{\inf_{i\in M}}h_i\le \epsilon$ dans $V$ pour tout $M\in \cal F$ contenant $M_0$.
On en d\'eduit que pour tout $M\in \cal F$ contenant $M_0$, on a
$h-h_i\le \epsilon$ dans $V$ pour tout $i\in M$. En appliquant le m\^eme proc\'ed\'e
aux fonctions $c-h_i$, $i\in I$,
 on peut trouver un voisinage fin $W$ de $x_0$, compact en topologie initiale et contenu dans $V$,
et un ensemble $M_1\in \cal F$, tel que $h_i-h<\epsilon$ dans $W$ pour tout
$M\supset M_1$ et tout $i\in M$. On a alors
$|h-h_i|<\epsilon$ dans $W$ pour tout $M\in \cal F$ contenant $ M_0\cup M_1$
et tout $i\in M$. Ainsi la famille $h_i$ converge
uniform\'ement selon $\cal F$ vers $h$ dans $W$.

\begin{cor}\label{cor2.3}
Soit $I$ un ensemble ordonn\'e et r\'eticul\'e
\`a droite et soit $(h_i)_{i\in I}$ une famille filtrante croissante de fonctions finement harmoniques dans
$U$. Si $(h_i)$ est uniform\'ement finement localement born\'ee dans $U$, alors $h=\sup_ih_i$ est
finement harmonique, et
$(h_i)$ est finement localement uniform\'ement convergente vers $h$
selon le filtre des sections de $I$.
\end{cor}

{\it D\'emonstration.} Il est clair que $h=\sup_{i\in I}h_i$ est
finement harmonique dans $U$. Soit $\cal F$ 
le filtre des sections de
$I$. Alors $(h_i)$ est convergente selon $\cal F$ vers la fonction $h=\sup_ih_i\in \cal S(U)$.
Le corollaire r\'esulte aussit\^ot du th\'eor\`eme pr\'ec\'edent.

\begin{cor}\label{cor2.4}  Soit $U$ un ouvert fin de
$\Omega$ et soit
$(h_n)$ une suite uniform\'ement finement localement born\'ee
(i.e. au sens de la topologie fine) de fonctions finement harmoniques dans $U$.
Alors on peut en extraire  une sous-suite $(h_{n_k})$
qui converge
uniform\'ement finement localement vers
une fonction finement harmonique $h$ dans $U$.
\end{cor}

{\it D\'emonstration.} Quitte \`a se placer finement localement et ajouter une constante
aux fonctions $h_n$, on peut supposer qu'il existe
une constante $c>0$ telle que $0\le h_n\le c$ pour tout entier $n$.
On peut alors extraire de $(h_n)$ une sous-suite $(h_{n_k})$ qui converge
au sens de la tolopologie naturelle de $\cal S(U)$ vers une fonction finement surharmonique
$h\ge 0$. Il suffit maintenant d'appliquer le th\'eor\`eme \ref{thm2.2}
\`a la famille $(h_{n_k})_k$ et le filtre des voisinages de $+\infty$ dans $\NN$.

\section{R\'egularit\'e du noyau de Green fin}

Rappelons que d'apr\`es \cite[Corollaire 2, p. 110]{EK1}, le c\^one $\cal S(U)$
muni de la topologie naturelle admet une base compacte.
Soit $B$ une base compacte de $\cal S(U)$
et soit $\Phi$ une forme lin\'eaire continue
positive sur $\cal S(U)$ d\'efinissant $B$, i.e $B=\{s\in \cal S(U): \Phi(s)=1\}$.
Pour tout $y\in U$, posons $P_y=\frac{G_U(.,y)}{\Phi(G_U(.,y))}$.
L'application $\varphi: U\longrightarrow B$ d\'efinie par
$\varphi(y)=P_y$ est injective. Nous identifions $U$ avec son image par $\varphi$. La topologie induite
sur $U$ par celle de $B$ sera appel\'ee la topolgie naturelle de $U$.
Cette topologie est ind\'ependante de la base $B$. En effet
si $B_1$ et $B_2$ sont deux bases compactes de $\cal S(U)$, alors $B_1$ et
$B_2$ sont hom\'eomorphes, donc les topologies induites sur $U$ par celles
de $B_1$ et $B_2$ sont identiques.

Nous commen\c cons d'abord par comparer la topolgie naturelle de $U$ avec la topologie
fine. La proposition suivante et ses deux corollaires ont \'et\'e d\'emontr\'es
dans \cite{EKF1}.
\begin{prop}\label{prop3.1} {\rm (}\cite[Corollary 3.15]{EKF1}{\rm)}
La fonction $U\ni y\mapsto P_y\in {\cal S}(U)$ est finement continue sur $U$.
\end{prop}

\begin{proof}

Soit $x_0,z\in U$, $x_0\neq z$, et $V$ un voisinage fin de $z$ tel que $x_0\notin {\overline V}$.
Par le principe de Harnack, on peut trouver un voisinage ouvert $W$ de $x_0$ tel que
$$(1-\epsilon)\inf_{y\in V}G(x_0,y)\le  \inf_{y\in V} G(x_0,y)\le (1+\epsilon)\inf_{y\in V}G(x_0,y)$$
pour tout $x\in W$. On en d\'eduit que
$$(1-\epsilon)\inf_{y\in V}G(x_0,y)\le  \widehat\inf_{y\in V} G(x_0,y)\le (1+\epsilon)\inf_{y\in V}G(x_0,y)$$
$x_0$ et $V$ \'etant arbitraire, on en d\'eduit que
${\rm f}$-$\lim{\widehat \inf}_{y\to z} \ G(.,y)= G(.,z)$ dans $U\setminus \{z\}$, donc partout.
Soit $\cal U$ un ultrafiltre sur $U$ plus fin que le filtre des
voisinages fins de $z$.
On a, dans ${\cal S}(U)$,
$$G(.,z)|U=\lim_{y,\cal U}G(.,y)|U=\lim_{y,\cal U}G_U(.,y)+\lim_{y,\cal U}\widehat R_{G(.,y)}^{\complement U}|U.$$
D'autre part, on a, pour tout $x\in U$,
$$(\lim_{y,\cal U}\widehat R_{G(.,y)}^{\complement U})(x)=(\sup_{M\in {\cal U}}\widehat \inf_{y\in M}\widehat R_{G(.,y)}^{\complement U})(x)
\le \lim_{y,\cal U}\widehat R_{G(.,y)}^{\complement U}(x)= \widehat R_{G(.,z)}^{\complement U}(x)$$
puisque la fonction $y\mapsto \widehat R_{G(.,y)}^{\complement U}(x)=
\widehat R_{G(.,x)}^{\complement U}(y)$ est finement continue
sur $U$. On en d\'eduit que $s=\lim_{y,\cal U} G_U(.,y)\ge G_U(.,z)>0$.
D'autre part on a aussi $s\le \lim_{y,\cal U}G(.,y)|U=G(.,z)|U$, donc $s\in \cal S(U)$.
On a \'egalement $G(.,y)-U=G_U(.,y)+\widehat R_{G(.,y)}^{\complement U}|U$,
d'o\`u par passage \`a la limite suivant $\cal U$:
$$G(.,z)|U=s+\lim{_\cal U}\widehat R_{G(.,y)}^{\complement U}|U\ge s+\widehat R_{G(.,y)}^{\complement U}|U,$$
ce qui prouve que $s$ est un potentiel fin finement harmonique dans
$U\setminus \{z\}$, donc de la forme $\alpha G_U(.,z)$ pour un certain $\alpha \in ]0,1]$
d'apr\`es \cite[Theorem, p. 203]{F0}. Il en r\'esulte que
$$\lim_{y,\cal U}P_y=\lim_{y,\cal U}\frac{G_U(.,y)}{\Phi(G_U(.,y))}=\frac{s}{\Phi(s)}=P_z.$$
On en d\'eduit finalement que ${\rm f}$-$\lim_{y\to z}P_y=P_z.$ Ce qui prouve bien
la fonction $U\ni y\mapsto P_y$ est  est finement continue
et termine la preuve.
\end{proof}

\begin{cor}\label{cor3.2} {\rm (}\cite[Corollary 3.15]{EKF1}${\rm)}$
La topologie naturelle de
$U$ est moins fine que la topologie fine de $U$.
\end{cor}

La proposition suivante n'est que le (a) de \cite[Lemma 3.14]{EKF1}.

\begin{prop}\label{prop3.4} La fonction $g$ d\'efinie sur $U$ par
$g(y)=\Phi(G_U(.,y))$ est finement continue sur $U$.
\end{prop}

{\it D\'emonstration.} Soient $z\in U$ et $\cal U$ un ultrafiltre plus fin que le filtre
des voisinages fins de $z$. On a $\lim\widehat{\inf}_{\cal U}G_U(.,y)\le \liminf_{\cal U} G(.,y)
\le G(.,z)$. Soit $V$ un ouvert fin tel que $V\subset {\overline V}\subset U$ et
$z\notin {\overline V}$. Il est clair que la famille $(G_U(.,y))_{y\in \complement V}$
est localement finement uniform\'ement
born\'ee dans $V$, donc d'apr\`es le th\'eor\`eme \ref{thm2.2} cette famille converge
uniform\'ement finement localement localement vers $G_U(.,z)$ dans $V$ puisque
pour tout $x\in U$, on a ${\rm f}-\lim_{y\to z}G_U(x,y)=G_U(x,z)$.  On en d\'eduit que
$\lim_{y,\cal U}G_U(.,y)=G_U(.,z)$ dans $U\setminus \{z\}$, donc partout
puisque $\{z\}$ est polaire. Donc $\lim_{\cal U} \Phi(G_U(.,y))=G_U(.,z)$.
Il en r\'esulte que ${\rm f}$-$\lim_{y\to z} g(y)={\rm f}$-$\lim_{y\to z}\Phi(G_U(.,y))=G_U(.,z)$.
Comme $z$ est arbitraire, on en d\'eduit
que  $g$ est finement continue sur $U$.

\begin{lemma}\label{lemma3.4}
Soit $\alpha >0$. Alors l'ensemble $A=\{y\in U: \Phi(G_U(.,y))\ge \alpha\}$ est compact en topologie
euclidienne.
\end{lemma}

{\it D\'emonstration.} Pour tout
$y\in A$, on a $\Phi(G(.,y)|U)=\Phi(G_U(.,y))+\Phi(\widehat R_{G(.,y)}^{\complement U}|U) \ge
\Phi(G_U(.,y))\ge \alpha$ et donc $A\subset \{y\in \Omega: \Phi(G(.,y)|U)\ge \alpha\}$. Or
$\liminf_{y\to z}G(.,y)=0$ pour tout $z\in \partial \Omega$
car $\Omega$ est r\'egulier, donc $A$ est relativement compact dans $\Omega$.
Soit $(y_n)$ une suite de points de $A$ qui converge vers
$y\in \overline\Omega$, alors $y\in \Omega$ d'apr\`es ce qui pr\'ec\`ede. D'autre part
on a
$G(.,y_n)|U=G_U(.,y_n)+\widehat R_{G(.,y_n)}^{\complement U}|U$ pour tout $n$, et on peut extraire
de $(y_n)$ une sous-suite $(y_{n_k})$ telle que les suites
$(G_U(.,y_{n_k}))$ et $(\widehat R_{G(.,y_{n_k})}^{\complement U}|U)$ convergent
dans $\cal S(U)$. On en d\'eduit en passant \`a la limite quand $k\to +\infty$ que
$G(.,y)|U=\lim_{k\to +\infty}G_U(.,y_{n_k})+\lim_{k\to +\infty}\widehat R_{G(.,y_{n_k})}^{\complement U}|U$,
donc $\lim_{k\to +\infty}G_U(.,y_{n_k})$ est finement harmonique $\ge 0$ dans
$U\setminus\{y\}$ puisque $G(.,y)|U$ est finement harmonique dans $U\setminus\{y\}$. De plus, on a
$\lim\inf R_{G(.,y_{n_k})}^{\complement U}|U\ge \widehat R_{G(.,y)}^{\complement U}|U$ et, en r\'egularisant, on obtient
$\lim_{k\to \infty} \widehat R_{G(.,y_{n_k})}^{\complement U}|U\ge \widehat R_{G(.,y)}^{\complement U}|U$ et,
par cons\'equent,  $\lim_{k\to +\infty}G_U(.,y_{n_k})\le G_U(.,y)$.
Donc la fonction $\lim_{k\to +\infty}G_U(.,y_{n_k})$ est un potentiel fin dans
$U$, finement harmonique dans $U\setminus \{y\}$. Il r\'esulte alors de \cite[Theorem, p.203]{F0}
que $\lim_{k\to +\infty}G_U(.,y_{n_k})=\gamma G_U(.,y)$ pour un certain
$\gamma \in [0,1]$. On en d\'eduit que $\Phi(G_U(.,y))\ge \gamma \Phi(G_U(.,y))=\lim_{k\to +\infty} \Phi(G_U(.,y_{n_k}))\ge \alpha$,
ce qui prouve que $y\in A$. Donc $A$ est compact en topologie euclidienne.\\

Avec les notations pr\'ec\'edentes d\'efinissons la fonction
$G_1$ sur $U\times U$ par $G_1(x,y)=P_y(x)$.

\begin{theorem}\label{thm3.5}
La fonction $G_1$ est s.c.i. sur  $U\times U$ et continue sur $(U\times U)\setminus D$,

\noindent o\`u $U$ est muni de la topologie fine, $U\times U$ est muni de la topologie produit correspondante et o\`u $D$ est
la diagonale de $U\times U$.
\end{theorem}

{\it D\'emonstration.} La fonction $G_1$ est s.c.i sur $U\times U$ d'apr\`es
le corollaire \ref{cor3.2} et \cite[Proposition 3.2, iii)]{EKF1}.
Soient $(x_0,y_0)\in U\times U$ tel que
$x_0\ne y_0$ et  $\alpha>0$ tel que  $\alpha<\Phi(G_U(.,y_0))$. Posons
$U_\alpha=\{y\in U: \Phi(G_U(.,y))>\alpha\}$. Alors $U_{\alpha}$ est un ouvert fin
d'apr\`es la proposition \ref{prop3.4} et on a $y_0\in U_{\alpha}$.
On peut trouver un voisinage fin $W$ de $y_0$ compact en topologie euclidienne tel que
$W\subset U_\alpha$, et un voisinage $V_1$ de $x_0$ ouvert en topologie euclidienne
tels que $V_1\cap W=\emptyset$.
Les fonctions
$G(.,y)$, $y\in W$ sont harmoniques dans $\Omega\setminus W$ et la fonction $G$ est continue
au point $(x_0,y_0)$, on peut donc trouver, en vertu de la propri\'et\'e de Harnack,
un voisinage ouvert $V$ (dans $\Omega$) de $x_0$ tel que $V\subset V_1$
et dans lequel les fonctions $\frac{G(.,y)}{\alpha}$, $y\in U_\alpha$, sont
major\'ees  par une m\^eme constante $C>0$.
D'autre part on a
$G_1(.,y)=\frac{G_U(.,y)}{\Phi(G_U(.,y))}\le \frac{G(.,y)}{\alpha}$ pour tout
$y\in W$, donc
$G_1(.,y)\le \frac{C}{\alpha}$ dans $V$ pour tout $y\in W$.
Soit $\cal V$ le filtre
des voisinages fins de $y_0$. Alors d'apr\`es la proposition \ref{prop3.1}, on a
$\lim_{\cal V}G_1(.,y)=G_1(.,y_0)$ dans $\cal S(U)$. En vertu du Th\'eor\`eme \ref{thm2.2} et ce qui pr\'ec\`ede
on a $\lim_{\cal V}G_1(.,y)=G_1(.,y_0)$ uniform\'ement dans
un voisinage fin $V_2$ de $x_0$ contenu dans $U\setminus W$.
On en d\'eduit que pour $\epsilon >0$ donn\'e, on a
$$|G_1(x,y)-G_1(x_0,y_0)|\le |G_1(x,y)-G_1(x,y_0)|+|G_1(x,y_0)-G_1(x_0,y_0)|\le \epsilon$$
dans le produit $U_1\times U_2$ d'un voisinage fin  de $x_0$ et d'un voisinage fin
de $y_0$. Il en r\'esulte bien que la fonction
$G_1$ est continue en $(x_0,y_0)$.

\begin{cor}\label{cor3.6}
La fonction $G_U$ est s.c.i. sur  $U\times U$ et continue sur $(U\times U)\setminus D$,

\noindent o\`u $U$ est muni de la topologie fine, $U\times U$ est muni de la
topologie produit correspondante et o\`u $D$ est
la diagonale de $U\times U$.
\end{cor}

{\it D\'emonstration.} En effet, on a
$G_U(x,y)=\Phi(G_U(.,y))G_1(x,y)$ pour tout $(x,y)\in U\times U$.
La fonction $(x,y)\mapsto \Phi(G_U(.,y))$ est continue sur $U^2$
d'apr\`es la proposition \ref{prop3.4}. Donc $G_U$
poss\`ede les propri\'et\'es requises d'apr\`es le th\'eor\`eme
\ref{thm3.5}.

\begin{remark}\label{remark5.7}
Le noyau de Green fin $G_U$ de $U$ est finement continu sur $U\times U$ en tant
qu'ouvert fin de $\RR^{2n}$. En effet, la fonction $G_U$ est s\'epar\'ement
finement surharmonique $\ge 0$, il r\'esulte alors de \cite[Th\'eor\`eme 4.5]{EK2}
que $G_U$ est finement surharmonique sur $U^2$, donc finement continue sur $U^2$.
\end{remark}


\thebibliography{99}

\bibitem{AG} Armitage, D. H., Gardiner, S. J.: \textit{Classical
Potential Theory}, Springer, London, 2001.

\bibitem{BBC} Boboc, N., Bucur, Gh., Cornea, A.: \textit{Order and convexity in Potential theory:
H-cones}, Lect. Notes in Math. 853, Springer-Verlag, 1981.

\bibitem{CC}  Constantinescu, C., Cornea, A.: \textit{Potential Theory on
Harmonic spaces}, Springer Verlag Heidelberg, 1972.

\bibitem{DM} Dellacherie, C., Meyer, P. A.: \textit{Probabilit\'es et
Potentiel}, Hermann, Paris 1987, Chap. XII-XVI.

\bibitem{DO} Doob, J.L.: \textit{Classical Potential Theory and its
Probabilistic Counterpart}, Springer-Verlag, Berlin, 2001.

\bibitem{EK1} El Kadiri, M.: \textit {Sur la d\'ecomposition de Riesz
et la repr\'esentation int\'egrale des fonctions finement surharmoniques}, Positivity 4 (2000), no. 2, 105--114.

\bibitem{EK2} El Kadiri, M.: \textit {Fonctions s\'epar\'ement
finement surharmoniques},   Positivity 7, no. 3 (2003), no. 3, 245--256.

\bibitem{EKF1} El Kadiri, M., Fuglede, B.: \textit {Martin boundary of a fine domain and a Fatou-Na\"im-Doob Theorem
 for finely superharmonic functions}, arXiv:1403.0857.

\bibitem{EKF2} El Kadiri, M., Fuglede, B.: \textit {Sweeping at the Martin boundary of a fine domain},
arXiv:1409.7098.

\bibitem{F1} Fuglede, B.: \textit {Finely harmonic functions}, Lecture Notes in Math., 289,
Springer-Verlag, 1972.

\bibitem{F2}Fuglede, B.: \textit {Localization in Fine Potential Theory and Uniform Approximation
by Subharmonic Functions}, J. Funct. Anal. 49 (1982), 52-72.

\bibitem{F0} Fuglede, B.: \textit {Sur la fonction de Green pour un
domaine fin}, Ann. Inst. Fourier \textbf{25}, 3--4 (1975), 201--206.

\bibitem{F3} Fuglede, B.: \textit {Repr\'esentation int\'egrale des
potentiels fins}, Comptes Rendus, 300,
S\'erie I (1985), 129 -132.

\bibitem{F5} Fuglede, B.: \textit {Finely harmonic mappings and finely holomorphic
functions}, Ann. Acad. Sci. Fenn. Serie A.I. Mathematica, Helsinki, 2, (1976) 113-127.

\bibitem{GH} Gardiner, S. J.,  Hansen W.: \textit {The Riesz decomposition of finely superharmonic functions.} Adv. Math. 214 (2007), no. 1, 417--436.

\bibitem{He}  Herv\'e, R.-M.: \textit{ Recherches axiomatiques sur la
th\'eorie des fonctions surharmoniques et du potentiel},
Ann. Inst. Fourier, 12 (1962), 415-571.

\bibitem{M}  Mokobodzki, G.: \textit {Repr\'esentation int\'egrale des fonctions surharmoniques
au moyen des r\'eduites}, Ann. Inst. Fourier, 15 (1965), 103-112.
\end{document}